\documentclass[11pt]{amsart}

\usepackage{amscd}
\usepackage{amsmath, amssymb, eucal, mathrsfs, verbatim}
\usepackage{amsfonts}

\newcommand{\de}{\partial} 
\newcommand{\db}{\overline{\partial}}

\newcommand{\Ric}{\mathrm{Ric}}

\newcommand{\ov}[1]{\overline{#1}}

\newcommand{\mn}{\sqrt{-1}}

\newcommand{\ti}[1]{\tilde{#1}}
\newcommand{\vp}{\varphi}

\newcommand{\ve}{\varepsilon}
\newcommand{\osc}{\textrm{osc}}
\newcommand{\diam}{\mathrm{diam}}
\newcommand{\vol}{\mathrm{Vol}}
\newcommand{\Rm}{\mathrm{Rm}}

\renewcommand{\leq}{\leqslant}
\renewcommand{\geq}{\geqslant}

\begin{document}
\newcounter{remark}
\newcounter{theor}
\setcounter{remark}{0}
\setcounter{theor}{1}
\newtheorem{claim}{Claim}
\newtheorem{theorem}{Theorem}[section]
\newtheorem{proposition}[theorem]{Proposition}
\newtheorem{question}{question}[section]
\newtheorem{lemma}[theorem]{Lemma}
\newtheorem{defn}[theorem]{Definition}

\newtheorem{corollary}{Corollary}[section]

\newenvironment{remark}[1][Remark]{\addtocounter{theorem}{1} \begin{trivlist}
\item[\hskip
\labelsep {\bfseries #1  \thetheorem.}]}{\end{trivlist}}

\title[K\"ahler-Einstein metrics on Fano surfaces]{K\"ahler-Einstein metrics on Fano surfaces}
\author{Valentino Tosatti}
 \address{Department of Mathematics \\ Columbia University \\ New York, NY 10027}

  \email{tosatti@math.columbia.edu}
\begin{abstract} 
We give an exposition of a result of G. Tian, which says that a Fano surfaces admits a K\"ahler-Einstein metric precisely when the Lie algebra of holomorphic vector fields is reductive.
\end{abstract}
\maketitle
\setcounter{equation}{0}
\section{Introduction}
A Fano surface (or Del Pezzo surface) is a compact K\"ahler manifold $M$ of complex dimension $2$ and with ample anticanonical bundle $K_M^{-1}=\Lambda^2T^{1,0}M$. This is equivalent to the requirement that the first Chern class $c_1(M)=c_1(K_M^{-1})$ be representable by a K\"ahler form.

The question that we want to address in these notes is: which Fano surfaces admit K\"ahler-Einstein metrics? Recall that these are just K\"ahler metrics $\omega$ with Ricci curvature equal to the metric:
\begin{equation}\label{ke}\Ric(\omega)=\omega.\end{equation}
Taking the cohomology class of \eqref{ke} we see that any K\"ahler-Einstein metric 
is cohomologous to $c_1(M)$. There are many known obstructions to the existence of a K\"ahler-Einstein metric, the oldest one being the following
\begin{theorem}[Matsushima \cite{Mat}]\label{matt} If a Fano manifold $M$ admits a K\"ahler-Einstein metric, then the Lie algebra $\mathfrak{h}(M)$ of the spaces of holomorphic vector fields on $M$ is reductive (i.e. it is the complexification of a compact real subalgebra).
\end{theorem}

On Fano surfaces, the converse to this theorem holds:
\begin{theorem}[Tian \cite{Ti2}]\label{main}
A Fano surface $M$ admits a K\"ahler-Einstein metric if and only if   $\mathfrak{h}(M)$ is reductive.
\end{theorem}

In these notes we will give some ideas of the complicated proof of Theorem \ref{main}. First of all, let us note that all Fano surfaces are classified:

\begin{theorem}[Del Pezzo, Theorem 5.16, p.125 in \cite{fr}]\label{dp}A Fano surface $M$ is biholomorphic to one of the following: $\mathbb{CP}^2, \mathbb{CP}^1\times\mathbb{CP}^1$ or the blowup of $\mathbb{CP}^2$ at $1\leq k\leq 8$ distinct points in general position (which means that no $3$ are collinear, no $6$ lie on a conic and if $k=8$ then not all of them are contained in a cubic which is singular at (at least) one of the points).
\end{theorem}
Using this, one can restate the main theorem as
\begin{theorem}[Tian \cite{Ti2}]\label{main2}
All Fano surfaces admit K\"ahler-Einstein metrics except the blowup of $\mathbb{CP}^2$ at $1$ or $2$ distinct points.
\end{theorem}
The reason is that one can easily compute, using the classification theorem \ref{dp}, that all the Fano surfaces have reductive $\mathfrak{h}(M)$ except the blowup of $\mathbb{CP}^2$ at $1$ or $2$ distinct points. For example, the identity components of the automorphism groups of the blowup of $\mathbb{CP}^2$ at $1$ and $2$ point are isomorphic to the groups of complex matrices
$$\left(\begin{array}{ccc} * & * & * \\ 0 & * & * \\ 0 & * & * \end{array}\right)/\mathbb{C}^*, \quad \left(\begin{array}{ccc} * & 0 & * \\ 0 & * & * \\ 0 & 0 & * \end{array}\right)/\mathbb{C}^*,$$
whose Lie algebras are not reductive. On the other hand if one blows up $\mathbb{CP}^2$ at $4$ or more general points then the resulting manifold has $\mathfrak{h}(M)=0$.\\

\noindent{\bf Acknowledgments. }I am grateful to Z. Lu, R. Seyyedali, J. Song, G. Sz\'ekelyhidi and B. Weinkove for very useful comments, and to D.H. Phong for encouragement. These notes were originally written for my talks at the Informal complex geometry and PDE seminar at Columbia University in December 2009. I wish to thank all the participants in the seminar. 
\setcounter{equation}{0}
\section{Preliminary reductions}
First of all, $\mathbb{CP}^2$ and $\mathbb{CP}^1\times\mathbb{CP}^1$ admit explicit K\"ahler-Einstein metrics, the Fubini-Study metrics. We have seen above that the blowup of $\mathbb{CP}^2$ at $1$ or $2$ points have nonreductive $\mathfrak{h}(M)$ and so they do not admit a K\"ahler-Einstein metric by Matushima's theorem \ref{matt}.

All Fano surfaces $M$ which are the blowup of $\mathbb{CP}^2$ at $3$ or $4$ points in general position are biholomorphic to each other: in fact, any $3$ (or $4$) points in $\mathbb{CP}^2$ in general position can be mapped to any other $3$ (or $4$) by a biholomorphism of $\mathbb{CP}^2$, which induces a biholomorphism of the two blowups. So we can talk about ``the blowup of $\mathbb{CP}^2$ at $3$ or $4$ points'', and these admit K\"ahler-Einstein metrics thanks to the work of Tian-Yau \cite[Theorem 3.3]{TY} (see theorem \ref{estima} below, and also Siu \cite{Si} and Nadel \cite{Nad} for the case of $3$ points).

The main theorem \ref{main2} is thus reduced to proving:
\begin{theorem}\label{main3}If $M$ is a blowup of $\mathbb{CP}^2$ at $5\leq k\leq 8$ points in general position, then $M$ admits a K\"ahler-Einstein metric.
\end{theorem}

To prove this theorem, we will use a continuity argument in the space of Fano manifolds diffeomorphic to the blowup of $\mathbb{CP}^2$ at $5\leq k\leq 8$ points. More precisely, for $5\leq k\leq 8$ let us define
$$\mathfrak{M}_k=\{k\textrm{-tuples of points in }\mathbb{CP}^2\textrm{ in general position}\}/\textrm{Aut}(\mathbb{CP}^2),$$
which is the same as the set of isomorphism classes of complex structures on the blowup of $\mathbb{CP}^2$ at $k$ points with positive first Chern class (we could similarly define $\mathfrak{M}_k$ for $1\leq k\leq 4$, but it would be just one point). The set $\mathfrak{M}_k$ is a (noncompact) complex manifold, with the natural induced complex structure, and is easily seen to be connected. We will often abuse notation and say ``$M$ in $\mathfrak{M}_k$'', meaning that $M$ is a Fano surface which is the blowup of $\mathbb{CP}^2$ at $k$ points in general position.

If we call $\mathcal{KE}_k$ the subset of $\mathfrak{M}_k$ of all complex structures that admit a K\"ahler-Einstein metric, then we need to show that $\mathcal{KE}_k=\mathfrak{M}_k$. We have the following results:
\begin{theorem}[Tian-Yau \cite{TY}]\label{tianyau} For each $5\leq k\leq 8$, the set $\mathcal{KE}_k$ is nonempty.
\end{theorem}
This is achieved by constructing K\"ahler-Einstein metrics on manifolds obtained by blowing up some sufficiently symmetric configuration of points (see also theorem \ref{estima}).
\begin{lemma}[Lemma 1.3 in \cite{Ti2}] For each $5\leq k\leq 8$, the set $\mathcal{KE}_k$ is open in $\mathfrak{M}_k$.
\end{lemma}
This is a simple consequence of the Implicit Function Theorem, since any manifold $M$ in $\mathfrak{M}_k$ has no nonzero holomorphic vector fields. Theorem \ref{main3} is then reduced to showing:
\begin{theorem}\label{main4} For each $5\leq k\leq 8$, the set $\mathcal{KE}_k$ is closed in $\mathfrak{M}_k$.
\end{theorem}
This is because we then have that $\mathcal{KE}_k$ is nonempty, open and closed in 
$\mathfrak{M}_k$, which is itself connected, and it must then coincide with $\mathfrak{M}_k$.
To show that $\mathcal{KE}_k$ is closed in $\mathfrak{M}_k$, we need to show that given any sequence of points $x_i$ in $\mathcal{KE}_k$ which converges to a point $x_\infty$ in $\mathfrak{M}_k$, we have in fact that $x_\infty$ is in $\mathcal{KE}_k$.
Each $x_i$ corresponds to a Fano surface $M_i$ in $\mathfrak{M}_k$ which admits a K\"ahler-Einstein metric $\omega_i$, while the point $x_\infty$ corresponds to a Fano surface $M_\infty$ in $\mathfrak{M}_k$. The fact that $x_i\to x_\infty$ means that the $k$-tuples of points corresponding to $M_i$ converge to the $k$-tuple of points corresponding to $M_\infty$. This implies that the complex manifolds $M_i$ converge to $M_\infty$ in the sense of Cheeger-Gromov, i.e. modulo modifying the complex structure $J_i$ of $M_i$ by diffeomorphisms (which we can and will assume are just the identity), we can assume that $J_i\to J_\infty$ smoothly (on the underlying differentiable manifold). Moreover one can find reference K\"ahler metrics $\ti{\omega}_i$ on $M_i$ cohomologous to $c_1(M_i)$ which converge smoothly to a K\"ahler metric $\ti{\omega}_\infty$ on $M_\infty$ cohomologous to $c_1(M_\infty)$. 

To see why these facts hold, one can for example easily construct a holomorphic map $X\to (\mathbb{CP}^2)^k\backslash \Sigma$ where $X$ is a quasiprojective variety so that the fiber over a $k$-tuple of points in $(\mathbb{CP}^2)^k$ in general position is the blow-up of $\mathbb{CP}^2$ at these points (and $\Sigma$ corresponds to configurations of points not in general position).
One then embeds $X$ in a large projective space, using the relative anticanonical bundle. The induced complex structures on the fibers will be $J_i$ and $J_\infty$ and restricting (a multiple of) the Fubini-Study metric gives the reference metrics $\ti{\omega}_i$ and $\ti{\omega}_\infty$.

If we could show that the K\"ahler-Einstein metrics $\omega_i$ (possibly modulo subsequences) converge smoothly to a limiting K\"ahler metric $\omega_\infty$ on $M_\infty$, then this would be automatically K\"ahler-Einstein and this would prove that indeed $x_\infty$ is in $\mathcal{KE}_k$.

Since the K\"ahler metrics $\omega_i$ and $\ti{\omega}_i$ are cohomologous, there are smooth functions $\vp_i$ such that $\omega_i=\ti{\omega}_i+\mn\de\db\vp_i$. The functions $\vp_i$ are only unique up to addition of a constant, and they can be normalized as follows: first we denote by $f_i$ the normalized Ricci potential of $\ti{\omega}_i$, which is defined by
$$\Ric(\ti{\omega}_i)=\ti{\omega}_i+\mn\de\db f_i,\quad \int_{M_i}e^{f_i}\ti{\omega}_i^2=\int_{M_i}\ti{\omega}_i^2=V,$$
where $V$ denotes the volume of $(M_i,\omega_i)$, which is the topological number
$$V=\int_{M_i}\ti{\omega}_i^2=c_1(M_i)^2=9-k.$$
Note that the functions $f_i$ converge smoothly to the Ricci potential $f_\infty$ of $\ti{\omega}_\infty$. Then the potentials $\vp_i$ can be normalized by imposing that they solve the complex Monge-Amp\`ere equation
\begin{equation}\label{ma}
(\ti{\omega}_i+\mn\de\db\vp_i)^2=e^{f_i-\vp_i}\ti{\omega}_i^2,
\end{equation}
which expresses the fact that the metrics $\omega_i$ are K\"ahler-Einstein.
We have the following theorem:
\begin{theorem} If there is a constant $C$ independent of $i$ such that
\begin{equation}\label{supest}
\sup_i\sup_{M_i}\vp_i\leq C,
\end{equation}
then $\mathcal{KE}_k$ is closed in $\mathfrak{M}_k$.
\end{theorem}
The reason for this is the following. First for each $i$ there is a Harnack inequality of the form
\begin{equation}\label{harn1}-\inf_{M_i}\vp_i\leq 2\sup_{M_i}\vp_i+C,
\end{equation}
where here and from now on $C$ denotes a constant independent of $i$, which might change from line to line. This is proved by Tian \cite{Ti1}, using the fact that each $M$ in $\mathfrak{M}_k$ with $5\leq k\leq 8$ has no nonzero holomorphic vector fields (see also Siu \cite{Si} for a slightly weaker statement which would also suffice). In fact, to prove \eqref{harn1}, Tian \cite[Proposition 2.3 (i)]{Ti1} first proves that
\begin{equation}\label{improved1}
\frac{1}{V}\int_{M_i}(-\vp_i)\omega_i^2\leq 2\sup_{M_i}\vp_i,
\end{equation}
(we will give a proof of this below in \eqref{harn2})
and then uses the Green formula for the K\"ahler-Einstein metrics $\omega_i$ (their Green functions have a uniform positive lower bound independent of $i$ because the metrics $\omega_i$ have bounded Sobolev and Poincar\'e constants, see Lemma \ref{moserr}) to get
\begin{equation}\label{har1}
-\inf_{M_i}\vp_i\leq \frac{1}{V}\int_{M_i}(-\vp_i)\omega_i^2+C,
\end{equation}
and combining \eqref{improved1} and \eqref{har1} gives \eqref{harn1}.
Granted this, \eqref{supest} would give a uniform bound on the oscillation of $\vp_i$,
$$\sup_i\osc_{M_i}\vp_i=\sup_i(\sup_{M_i}\vp_i-\inf_{M_i}\vp_i)\leq C.$$
Moreover \eqref{ma} immediately gives that
$$\sup_{M_i}\vp_i\geq 0,\quad \inf_{M_i} \vp_i\leq 0,$$
and so we get a bound on the $L^\infty$ norm of $\vp_i$
$$\sup_i\sup_{M_i} |\vp_i|\leq C.$$
Then the $C^2$ and higher order estimates of Yau \cite{Y, Si} for the complex Monge-Amp\`ere equation \eqref{ma} give uniform bounds
$$\sup_i\|\vp_i\|_{C^\ell(M_i,\ti{\omega}_i)}\leq C,$$
for all $\ell\geq 0,$ 
and so by the Ascoli-Arzel\`a theorem a subsequence of the metrics $\omega_i$ converges smoothly to a K\"ahler-Einstein metric $\omega_\infty$ on $M_\infty$.

The main theorem \ref{main4} is thus reduced to establishing the estimate \eqref{supest}. Before we move on to the proof of \eqref{supest}, let us introduce Tian's $\alpha$-invariants, which are central to the proof.
\setcounter{equation}{0}
\section{$\alpha$-invariants}
We start with the following result
\begin{proposition}[H\"ormander \cite{Ho}, Tian \cite{Ti1}]
If $(M,\omega)$ is a compact K\"ahler manifold, then there exist constants $C,\alpha>0$ so that for any K\"ahler potential $\vp$ (i.e. any smooth real-valued function $\vp$ with $\omega+\mn\de\db\vp>0$) we have
\begin{equation}\label{alph1}\int_M e^{-\alpha(\vp-\sup_M\vp)}\omega^n\leq C.
\end{equation}
\end{proposition}

Using this, Tian defined the $\alpha$-invariant as follows
\[\begin{split}\alpha(M,\omega)=\sup\{\alpha>0\ |\ &\exists C>0 \textrm{ such that }\eqref{alph1} \textrm{ holds}\\
&\textrm{ for all K\"ahler potentials }\vp\}.\end{split}\]
It is immediate to see that if $\omega'$ is another K\"ahler form cohomologous to $\omega$ then 
$\alpha(M,\omega)=\alpha(M,\omega')$. When the cohomology class is $c_1(M)$, we will simply denote this by $\alpha(M)$. It is also immediate to see that $\alpha(M)$ is invariant under biholomorphisms.

If $G$ is a compact subgroup of the automorphism group of $M$, we can consider $\alpha_G(M)$ where one restricts to $G$-invariant $\omega$ and $\vp$. In particular $\alpha_G(M)\geq\alpha(M)$. We have the following useful result
\begin{theorem}[Tian \cite{Ti1}]\label{alpha2} If on an $n$-dimensional Fano manifold $M$ we have
$$\alpha(M)>\frac{n}{n+1},$$
then $M$ admits a K\"ahler-Einstein metric. The same result holds for $\alpha_G(M)$.
\end{theorem}
So on a Fano surface $M$ we get a K\"ahler-Einstein metric provided $\alpha(M)>2/3$.
We have the following result, which is a combination of computations of Tian-Yau \cite{TY}, Tian \cite{Ti2}, Song \cite{So}, Cheltsov \cite{Ch}, Heier \cite{He} and Shi \cite{Sh} (see also \cite[Remark 1.1]{He}, \cite{CW} and \cite{CG}).
\begin{theorem}\label{estima}If $M$ is a Fano surface have the following estimates:
\begin{itemize}\item If $M$ is in $\mathfrak{M}_8$, then $\alpha(M)\geq 5/6,$
\item If $M$ is in $\mathfrak{M}_7$, then $\alpha(M)\geq 3/4,$
\item If $M$ is in $\mathfrak{M}_5$ there exists $G\subset \mathrm{Aut}(M)$ compact with $\alpha_G(M)\geq 1,$
\item If $M$ is in $\mathfrak{M}_4$, there exists $G\subset \mathrm{Aut}(M)$ compact with $\alpha_G(M)\geq 1,$
\item If $M$ is in $\mathfrak{M}_3$, there exists $G\subset \mathrm{Aut}(M)$ compact with $\alpha_G(M)\geq 1,$
\item If $M$ is in $\mathfrak{M}_6$, then $\alpha(M)\geq 2/3.$ There exist manifolds $M$ in  $\mathfrak{M}_6$ with $\alpha(M)\geq 3/4$, and there exist also manifolds with $\alpha_G(M)=2/3$ for any  $G\subset \mathrm{Aut}(M)$ compact.
\end{itemize}
\end{theorem}
This theorem, parts of which came after Tian's work \cite{Ti2}, shows that except for $\mathfrak{M}_6$, we can directly apply Theorem \ref{alpha2} to prove the main theorem \ref{main3}. However, to cover the case of $\mathfrak{M}_6$, we will need all of the analysis that follows (which covers all the cases of $\mathfrak{M}_k$ with $5\leq k\leq 8$).

We now define some kind of ``algebraic $\alpha$-invariants'' using plurianticanonical sections. For any Fano manifold $M$ and any $m\geq 1$ we let $H^0(K_{M}^{-m})$ be the space of global $m$-anticanonical sections of $M$, which is a finite dimensional vector space of dimension $N_m$ (which for $m$ large can be computed using the Riemann-Roch formula). We also fix a Hermitian metric $h$ on the fibers of the line bundle $K_M^{-1}$ with positive curvature $\omega$, a K\"ahler form in $c_1(M)$.
We also have the induced metric $h^m$ on $K_M^{-m}$ with curvature $m\omega$. If $S\in H^0(K_{M}^{-m})$, its pointwise norm $|S|^2_{h^m}$ is a smooth nonnegative function on $M$.
We then let
\[\begin{split}\alpha_{m,1}(M)=\sup\bigg\{\alpha>0\ \bigg|\ &\exists C>0 \textrm{ with } \int_M (|S|^2_{h^m})^{-\alpha/m}\omega^n\leq C\\
&\textrm{ for all }S\in H^0(K_{M}^{-m}), \int_M |S|^2_{h^m}\omega^n=1\bigg\},\end{split}\]
and also
\[\begin{split}\alpha_{m,2}(M)=\sup\bigg\{\alpha>0\ &\bigg|\ \exists C>0 \textrm{ with } \int_M (|S_1|^2_{h^m}+|S_2|^2_{h^m})^{-\alpha/m}\omega^n\leq C\\
&\textrm{ for all }S_1,S_2\in H^0(K_{M}^{-m}), \int_M \langle S_i,S_j\rangle_{h^m}\omega^n=\delta_{ij}\bigg\}.\end{split}\]
It is clear that 
\begin{equation}\label{triv1}
\alpha_{m,2}(M)\geq \alpha_{m,1}(M),
\end{equation}
and (using the techniques in the Appendix) one can also see that 
\begin{equation}\label{triv2}
\alpha_{m,1}(M)\geq \alpha(M). 
\end{equation}
Moreover, these invariants do not depend on the choice of $\omega, h$ (one can also define their $G$-invariant counterparts, but we won't need them).
We have the following crucial result due to Tian \cite{Ti2} (see also \cite{Sh, CW})
\begin{theorem}[Tian \cite{Ti2}]\label{alpha3} For any $M$ in $\mathfrak{M}_6$ and for any $m\geq 1$ we have 
$$\alpha_{m,2}(M)>2/3.$$
\end{theorem}
Tian also proves that for any $M$ in $\mathfrak{M}_5$ we have $\alpha_{m,2}(M)\geq 3/4$, but we won't need this since we now have a better estimate from Theorem \ref{alpha2}. In fact, in \cite{Ti2} these are proved for $m$ any multiple of $6$, and the general case follows from \cite{Sh}.

In the Appendix we will outline how one can compute these $\alpha$-invariants, using algebraic geometric methods.
\setcounter{equation}{0}
\section{Outline of the main argument}
To summarize what we did so far, we assume that $5\leq k\leq 8$, and we need to prove the estimate \eqref{supest}. In fact, by Theorem \ref{alpha2}, we only need to consider the case $k=6$, but we will follow Tian's original presentation and consider all cases $5\leq k\leq 8$.

Let us first define the $m^{th}$ density of states function: if $M$ is any Fano manifold and $h, \omega$ are as in the previous section, then for any $m\geq 1$ we can define 
$$\rho_m(\omega)=\sum_{j=1}^{N_m} |S_j|^2_{h^m},$$
where $S_1,\dots,S_{N_m}$ are a basis of $H^0(K_M^{-m})$ which is orthonormal with respect to the $L^2$ inner product $\int_M\langle S_1,S_2\rangle_{h^m}\omega^n$. Clearly $\rho_m(\omega)$ is independent of the choice of basis, and is also unchanged if we scale $h$ by a constant. The integral  $\int_M \rho_m(\omega)\omega^n$ equals $N_m$, the dimension of $H^0(K_M^{-m})$. Moreover if $m$ is sufficiently large so that $K_M^{-m}$ is very ample, then $\rho_m(\omega)$ is strictly positive on $M$.

We now apply this construction to the K\"ahler-Einstein metrics $\omega_i$ and get functions $\rho_m(\omega_i)$.
\begin{defn}
We say that a ``partial $C^0$ estimate'' holds if there exist $m_0\geq 1$ and $c>0$ such that
\begin{equation}\label{partialc0}
\inf_i\inf_{M_i}\rho_{m_0}(\omega_i)\geq c>0.
\end{equation}
\end{defn}

We will explain the reason for this name later (in proposition \ref{partc0}). The proof of the estimate \eqref{supest} then proceeds in three (independent) steps.\\

\noindent{\bf Step 1. }A partial $C^0$ estimate holds.\\

\noindent{\bf Step 2. }If a partial $C^0$ estimate holds and if for any $M$ in $\mathfrak{M}_k$ we have
$\alpha_{m_0,1}(M)>2/3$, then \eqref{supest} holds.\\

\noindent{\bf Step 3. }If a partial $C^0$ estimate holds and if for any $M$ in $\mathfrak{M}_k$ we have
$\alpha_{m_0,1}(M)=2/3$ and $\alpha_{m_0,2}(M)>2/3$ then \eqref{supest} holds.\\

By combining these three steps together with Theorem \ref{estima} and Theorem \ref{alpha3} (remembering \eqref{triv1}, \eqref{triv2}), we see that in all the cases $5\leq k\leq 8$ the estimate \eqref{supest} holds, and we are done.
\setcounter{equation}{0}
\section{Step 1 - Orbifold Compactness}
In this section we will prove that a partial $C^0$ estimate holds. Before doing that, let us explain the reason for calling it a partial $C^0$ estimate. Let us fix a hermitian metric $\ti{h}_i$ on $K_{M_i}^{-1}$ with curvature $\ti{\omega}_i$ (the reference metric), and let $h_i=\ti{h}_ie^{-\vp_i}$, which is a metric on $K_{M_i}^{-1}$ with curvature $\omega_i$ (the K\"ahler-Einstein metric). These induce metrics $\ti{h}_i^m, h_i^m$ on $K_{M_i}^{-m}$. Notice that, for $m$ large, the dimension $H^0(K_{M_i}^{-m})$ is equal to $N_m$ independent of $i$ (since it is computed by the Riemann-Roch formula in terms of characteristic numbers that depend only on $k$).

\begin{proposition}\label{partc0}
If a partial $C^0$ estimate holds then there exist a constant $C$, sequences of real numbers
$0<\lambda^i_1\leq\dots\leq\lambda^i_{N_m}=1$ and a sequence of bases $\{\ti{S}^i_j\}_{1\leq j\leq N_m}$ of $H^0(K_{M_i}^{-m})$ with $\int_{M_i}\langle \ti{S}^i_p,\ti{S}^i_q\rangle_{\ti{h}_i^m}\ti{\omega}_i^2=\delta_{pq}$ such that
\begin{equation}\label{partiall}
\sup_i\sup_{M_i}\left|\vp_i-\sup_{M_i}\vp_i-\frac{1}{m}\log\sum_{j=1}^{N_m}|\lambda^i_j|^2 
|\ti{S}^i_j|^2_{\ti{h}_i^m}\right|\leq C,
\end{equation}
where $m=m_0$.
\end{proposition}
In the limit when $i$ goes to infinity, the sections $\ti{S}^i_j$ converge smoothly to a basis of sections on $M_\infty$ (see the proof of Step 2 below), and some of the $\lambda^i_j$ converge to zero. The intersection of the zero loci of the limit sections with limit coefficient positive will in general be a nonempty subvariety of $M_\infty$,
and \eqref{partiall} says that $\vp_i-\sup_{M_i}\vp_i$ blows up precisely along this subvariety in $M_i$ (which is diffeomorphic to $M_\infty$), thus the name ``partial $C^0$ estimate''.

Notice that the functions $\frac{1}{m}\log\sum_{j=1}^{N_m}|\lambda^i_j|^2 
|\ti{S}^i_j|^2_{\ti{h}_i^m}$ are also K\"ahler potentials for the metric $\ti{\omega}_i$ because we have
\begin{equation}\label{algpot}\ti{\omega}_i+\mn\de\db \frac{1}{m}\log\sum_{j=1}^{N_m}|\lambda^i_j|^2 
|\ti{S}^i_j|^2_{\ti{h}_i^m}=\frac{\iota^*\tau^*\omega_{FS}}{m}>0,
\end{equation}
where $\iota:M_i\to\mathbb{CP}^{N_m-1}$ is the Kodaira embedding map given by the sections $\{\ti{S}^i_j\}_{1\leq j\leq N_m},$ the map $\tau:\mathbb{CP}^{N_m-1}\to \mathbb{CP}^{N_m-1}$ is the automorphism induced by the diagonal matrix with entries $\{\lambda^i_j\}_{1\leq j\leq N_m}$, and $\omega_{FS}$ is the Fubini-Study metric on $\mathbb{CP}^{N_m-1}$. The functions $\frac{1}{m}\log\sum_{j=1}^{N_m}|\lambda^i_j|^2 
|\ti{S}^i_j|^2_{\ti{h}_i^m}$ are sometimes referred to as ``algebraic K\"ahler potentials'' in the literature, and so a partial $C^0$ estimate says that we can uniformly approximate the potentials $\vp_i$ of the K\"ahler-Einstein metrics with algebraic potentials.

Before we prove this, we have the following lemma
\begin{lemma}\label{moserr} For any $m\geq 1$ there is a constant $C$ that depends only on $m$ and $k$ so that for all $i$ we have
$$\rho_m(\omega_i)\leq C.$$
\end{lemma}
\begin{proof}
Let $S$ be any holomorphic section of $K_{M_i}^{-m}$.
First, one easily computes that in general
\begin{equation}\label{moser}
\Delta_{\omega_i}|S|^2_{h_i^m}=|\nabla S|^2_i -2m|S|^2_{h_i^m}\geq -2m|S|^2_{h_i^m},
\end{equation}
where $\Delta_{\omega_i}$ is the Laplacian of $\omega_i$. 
Next, recall that the volume of $(M_i,\omega_i)$ is equal to the topological number $$V=\int_{M_i}\omega_i^2=c_1(M_i)^2=9-k.$$
Also since $\Ric(\omega_i)=\omega_i$, Myers' Theorem implies that the diameter of $(M_i,\omega_i)$ is bounded above by $\sqrt{3}\pi$. A classical result of Croke \cite{Cr} and Li \cite{li} then show that $(M_i,\omega_i)$ has a uniform bound on the Sobolev constant, that depends only on $k$ (in general the Sobolev constant bound for a Riemannian metric on a closed manifold depends on lower bounds for the volume and for the Ricci curvature and on an upper bound for the diameter). One can then apply the standard Moser iteration method to the differential inequality \eqref{moser} to get
\begin{equation}\label{supbd}
\sup_{M_i}|S|^2_{h_i^m}\leq C\int_{M_i}|S|^2_{h_i^m}\omega_i^2=C,
\end{equation}
where $C$ depends only on $k$ and $m$. Taking now an orthonormal basis of sections and summing we get the result.
\end{proof}
\begin{proof}[Proof of Proposition \ref{partc0}]
Thanks to Lemma \ref{moserr} we know that a partial $C^0$ estimate is equivalent to an estimate 
\begin{equation}\label{mos1}
\sup_i\sup_{M_i}|\log\rho_m(\omega_i)|\leq C,
\end{equation}
where here $m=m_0$.
We now take a basis $\{S^i_j\}_{1\leq j\leq N_m}$ of $H^0(K_{M_i}^{-m})$ with $$\int_{M_i}\langle S^i_p,S^i_q\rangle_{h_i^m}\omega_i^2=\delta_{pq},$$ and notice that since $h_i^m=e^{-m\vp_i}\ti{h}_i^m$ we clearly have
$$\vp_i=\frac{1}{m}\log\frac{\sum_{j=1}^{N_m}|S^i_j|^2_{\ti{h}^m_i}}{\sum_{j=1}^{N_m}|S^i_j|^2_{h^m_i}},$$
which is equivalent to
\begin{equation}\label{mos2}
\vp_i-\frac{1}{m}\log \sum_{j=1}^{N_m}|S^i_j|^2_{\ti{h}^m_i}=-\frac{1}{m}\log\rho_m(\omega_i).
\end{equation}
It follows from \eqref{mos1} and \eqref{mos2} that a partial $C^0$ estimate is equivalent to an estimate
$$\sup_i\sup_{M_i}\left|\vp_i-\frac{1}{m}\log \sum_{j=1}^{N_m}|S^i_j|^2_{\ti{h}^m_i} \right|\leq C.$$
We now choose another basis $\{\ti{S}^i_j\}_{1\leq j\leq N_m}$ of $H^0(K_{M_i}^{-m})$ with $$\int_{M_i}\langle \ti{S}^i_p,\ti{S}^i_q\rangle_{\ti{h}_i^m}\ti{\omega}_i^2=\delta_{pq},$$ and up to modifying $S^i_j$ and $\ti{S}^i_j$ by unitary transformations, we can assume that 
$$S^i_j=\mu^i_j \ti{S}^i_j,$$
for some positive real numbers $\mu^i_j$, with $0<\mu^i_1\leq\dots\leq \mu^i_{N_m}$. We then let $\lambda^i_j=\mu^i_j/\mu^i_{N_m}$ and we get
\begin{equation}\label{part1}
\sup_i\sup_{M_i}\left|\vp_i-\frac{2}{m}\log\mu^i_{N_m}-\frac{1}{m}\log \sum_{j=1}^{N_m}|\lambda^i_j|^2|\ti{S}^i_j|^2_{\ti{h}^m_i} \right|\leq C.
\end{equation}
We now claim that if a partial $C^0$ estimate holds, then we also
have
\begin{equation}\label{part2}
\sup_i\sup_{M_i}\left|\frac{2}{m}\log\mu^i_{N_m}-\sup_{M_i}\vp_i \right|\leq C.
\end{equation}
Once this is proved, combining \eqref{part1} and \eqref{part2} we get \eqref{partiall}. To prove \eqref{part2}, first use \eqref{supbd} to get
$$C\geq \sup_{M_i}|S^i_{N_m}|^2_{h_i^m}=|\mu^i_{N_m}|^2 \sup_{M_i}|\ti{S}^i_{N_m}|^2_{\ti{h}^m_i}e^{-m\sup_{M_i}\vp_i},$$
and the fact $\int_{M_i}|\ti{S}^i_{N_m}|^2_{\ti{h}^m_i}\ti{\omega}_i^2=1$
implies that $\sup_{M_i}|\ti{S}^i_{N_m}|^2_{\ti{h}^m_i}\geq 1/V,$
and so
$$\sup_i\sup_{M_i}\left(\frac{2}{m}\log\mu^i_{N_m}-\sup_{M_i}\vp_i\right)\leq C.$$
On the other hand the partial $C^0$ estimate \eqref{partialc0} implies that
\begin{equation}\label{part3}
0<c\leq \rho_{m}(\omega_i)=\sum_{j=1}^{N_m}|S^i_j|^2_{h_i^m}\leq
|\mu^i_{N_m}|^2\sum_{j=1}^{N_m}|\ti{S}^i_j|^2_{\ti{h}_i^m}e^{-m\vp_i},
\end{equation}
and arguing as in Lemma \ref{moserr} we can show that
\begin{equation}\label{c0norm}
\sup_{i,j} \sup_{M_i}|\ti{S}^i_j|^2_{\ti{h}^m_i}\leq C.
\end{equation}
In fact, the metrics $\ti{\omega}_i$ converge smoothly and so in particular they have uniform bounds on their scalar curvature and Sobolev constant. This implies that for any $i,j$ we have
$$\Delta_{\ti{\omega}_i}|\ti{S}^i_j|^2_{\ti{h}^m_i}\geq -C|\ti{S}^i_j|^2_{\ti{h}^m_i}.$$
Moser iteration then proves \eqref{c0norm}. This together with \eqref{part3}, evaluated at the point where $\vp_i$ achieves its maximum,
gives the reverse inequality
$$\sup_i\sup_{M_i}\left(\sup_{M_i}\vp_i-\frac{2}{m}\log\mu^i_{N_m}\right)\leq C,$$
which completes the proof of \eqref{part2}.
\end{proof}

We now outline the basic ideas in the proof of the partial $C^0$ estimate \eqref{partialc0}, referring to \cite{Ti2} for the details.
We want to prove that we have the estimate
\[\inf_i\inf_{M_i}\rho_{m_0}(\omega_i)\geq c>0,\]
for some constants $m_0, c>0$. If this did not hold, then for any given $m$ we could find a subsequence (still denoted by $i$) and points $x_i\in M_i$ so that
\begin{equation}\label{contr3}
\rho_m(\omega_i)(x_i)\to 0.
\end{equation}
This is proved in three steps:
\begin{theorem}[Orbifold Compactness]\label{comp1}
If $(M_i,\omega_i)$ is a sequence of K\"ahler-Einstein surfaces in $\mathfrak{M}_k$, then a subsequence converges (in the sense described in theorem \ref{comp3}) to a K\"ahler-Einstein orbifold surface $(X,\omega_\infty)$.
\end{theorem}
We won't give here the formal definition of orbifolds (see \cite{Ti2}), but we will just remark that for a K\"ahler-Einstein orbifold $(X,\omega_\infty)$ as above one can define the density of states function $\rho_m(\omega_\infty)$ using orbifold sections of the orbifold plurianticanonical bundle. By applying H\"ormander's $L^2$ estimates for the $\db$ operator \cite{Ho}, Tian proves:
\begin{theorem}[Tian \cite{Ti2}]\label{limin}
For any $m\geq 0$, if a sequence $(M_i,\omega_i)$ of K\"ahler-Einstein surfaces in $\mathfrak{M}_k$ converges to a K\"ahler-Einstein orbifold surface $(X,\omega_\infty)$, then we have
\begin{equation}\label{contr1}\liminf_{i\to\infty}\inf_{M_i}\rho_m(\omega_i)\geq \inf_X\rho_m(\omega_\infty).\end{equation}
\end{theorem}
Finally we have the following proposition, which is also proved using the $L^2- \db$ estimates (it can also be proved using algebraic geometry).
\begin{proposition}\label{comp2}
If a sequence $(M_i,\omega_i)$ of K\"ahler-Einstein surfaces in $\mathfrak{M}_k$ converges to a K\"ahler-Einstein orbifold surface $(X,\omega_\infty)$, then there is a positive integer $m$ that depends only on $k$ so that
\begin{equation}\label{contr2}\inf_X\rho_m(\omega_\infty)>0.\end{equation}
\end{proposition}
To prove the partial $C^0$ estimate \eqref{partialc0} it now suffices to apply theorems \ref{comp1} and \ref{limin} and choose $m_0=m$ as in proposition \ref{comp2}, since then \eqref{contr1}, \eqref{contr2}  contradict \eqref{contr3}. In fact the number $m_0$ can be made explicit as a function of $k$.

We won't prove these results here, but we will show how the compactness theorem \ref{comp1} fits into a more general result:
\begin{theorem}[Anderson \cite{and}, Bando-Kasue-Nakajima \cite{bkn}, Tian \cite{Ti2}]\label{comp3}
If $(M_i, g_i)$ is a sequence of compact real $n$-dimensional Einstein manifolds with the same Einstein constant (equal to $-1, 0$ or $1$), such that there are constants $D,V,R>0$ with
$$\diam(M_i, g_i)\leq D,$$
$$\vol(M_i, g_i)\geq V,$$
$$\int_{M_i}|\Rm(g_i)|^{n/2}_{g_i} dV_{g_i}\leq R,$$
then there exist a subsequence (still denoted by $i$) and a compact Einstein orbifold $(X, g_\infty)$ with singular set $S=\{x_1,\dots,x_\ell\}$ so that
the manifolds $(M_i,g_i)$ converge to $(X,g_\infty)$ in the sense of Gromov-Hausdorff, and moreover there are diffeomorphisms with the image
$F_i:X\backslash S\to M_i$ so that 
$F_i^*g_i$ converges to $g_\infty$ in $C^\infty_{\mathrm{loc}}(X\backslash S)$.
The number $\ell$ of singular points and the orders of all the local uniformization groups at the singular points are bounded by a constant that depends only on $n,D,V,R$. 
\end{theorem}
This theorem can be used directly to prove theorem \ref{comp1}, since we have already remarked that for the K\"ahler-Einstein surfaces $(M_i,\omega_i)$ in $\mathfrak{M}_k$ we have
$$\vol(M_i,\omega_i)=9-k,$$
$$\diam(M_i,\omega_i)\leq \sqrt{3}\pi,$$
and we also have the well-known formula
$$\int_{M_i}|\Rm(\omega_i)|^2_{\omega_i} \omega_i^2=\int_{M_i}|\Ric(\omega_i)|^2_{\omega_i} \omega_i^2+c_2(M_i)=2(9-k)+3+k.$$
We can thus apply theorem \ref{comp3}, and it is also clear that the limit orbifold $(X,g_\infty)$ is a K\"ahler orbifold. The fact that the bound on the number of singular points and on the orders of the uniformization groups depends only on $k$ is then used in proposition \ref{comp2} to show that $m_0=m$ depends only on $k$.

This completes the outline of the proof of the partial $C^0$ estimate \eqref{partialc0}. 

\begin{remark} To imitate this proof in higher dimension, one would need a bound like 
\begin{equation}\label{stronger}
\int_{M_i}|\Rm(\omega_i)|^n_{\omega_i} \omega_i^n\leq C,
\end{equation}
but this does not follow from the K\"ahler-Einstein condition (unless $n=2$),
and it is in fact much stronger than the known bound
$$\int_{M_i}|\Rm(\omega_i)|^2_{\omega_i} \omega_i^n\leq C.$$
In fact, Tian \cite{Ti3} proves that if one assumes \eqref{stronger} and $n\geq 3$ (and the K\"ahler-Einstein constant is $+1$), then the limit orbifold $(X,g_\infty)$ is in fact a smooth manifold (see also \cite[p.201]{Jo} for a sketch of another proof of this fact using algebraic geometric ingredients), although the convergence of $M_i$ to $X$ still happens only away from a finite number of points.
\end{remark}

\begin{remark} 
One can define a notion of partial $C^0$ estimate also for the K\"ahler-Ricci flow on a Fano manifold $M$, by requiring that the metrics $\omega_t$ along the flow ($t\geq 0$) satisfy
$$\inf_{t\geq 0}\inf_M\rho_m(\omega_t)\geq c>0,$$
for some fixed $m,c>0$. One can then easily show (as in proposition \ref{partc0}, using that the Sobolev constant of $\omega_t$ is uniformly bounded)
that a partial $C^0$ estimate implies an estimate of the form \eqref{partiall}. In general it is unknown whether such a partial $C^0$ estimate always holds, but it is rather easy to see that it holds if the sectional curvature remains bounded along the flow (see e.g. \cite{To}). According to Chen-Wang \cite{CW} such a partial $C^0$ estimate holds for the K\"ahler-Ricci flow on Fano surfaces.
\end{remark}
\setcounter{equation}{0}
\section{Step 2 - Semicontinuity of Complex Singularity Exponents}
In this section we will prove the second step in the proof of the main theorem, namely that if a partial $C^0$ estimate holds and if for any $M$ in $\mathfrak{M}_k$ we have $\alpha_{m_0,1}(M)>2/3$ then the estimate \eqref{supest} holds. For simplicity, we will write $m=m_0$.

For this we will need the following result, which is proved in the appendix of \cite{Ti2}. For different proofs of this and more general results see Phong-Sturm \cite{PS} and Demailly-Koll\'ar \cite{DK}. 

For each $i$ let $S_i$ be a global holomorphic section of $K_{M_i}^{-m}$. Recall that the complex surfaces $(M_i,\ti{\omega}_i)$ converge smoothly 
to the complex surface $(M_\infty,\ti{\omega}_\infty)$. We will assume that the sections $S_i$ converge smoothly to a section $S_\infty$ of $K_{M_\infty}^{-m}$, which is necessarily holomorphic, and which we assume is not identically zero.
\begin{theorem}[Semicontinuity of complex singularity exponents \cite{Ti3, PS, DK}]\label{semic}
In this case if $\beta>0$ is such that
$$\int_{M_\infty}|S_\infty|^{-\beta}_{\ti{h}_\infty^m}\ti{\omega}_\infty^2<\infty,$$
then for any $0<\alpha<\beta$ we have
\begin{equation}\label{semicc}
\lim_{i\to\infty}\int_{M_i}|S_i|^{-\alpha}_{\ti{h}_i^m}\ti{\omega}_i^2=
\int_{M_\infty}|S_\infty|^{-\alpha}_{\ti{h}_\infty^m}\ti{\omega}_\infty^2<\infty.
\end{equation}
\end{theorem}
Using this, we can easily finish the proof of Step 2.
\begin{proof}[Proof of Step 2]
Consider the sections $\ti{S}^i_j$ given by the partial $C^0$ estimate (Step 1). Since they are orthonormal, the $C^0$ norm of $|\ti{S}^i_j|^2_{\ti{h}_i^m}$ is bounded (see \eqref{c0norm}). In local homorphic coordinates, the sections $\ti{S}^i_j$ are represented by holomorphic functions, which are uniformly bounded in $L^\infty$ (since the metrics $\ti{h}_i^m$ are bounded). Cauchy's 
integral formula shows that locally we have uniform bounds on all the derivatives of $\ti{S}^i_j$, and so a subsequence of the sections $\ti{S}^i_j$ converges smoothly to a basis of sections $\{\ti{S}^\infty_j\}$ of $K_{M_\infty}^{-m}$, orthonormal with respect to the $L^2$ inner product defined using $\ti{h}_\infty^m$ and $\ti{\omega}_\infty^2$.
For any $\alpha>0$ we compute, using the partial $C^0$ estimate in the form \eqref{partiall}
and the fact that $\lambda^i_{N_m}=1,$
\[\begin{split}\int_{M_i}e^{-\alpha(\vp_i-\sup_{M_i}\vp_i)}\ti{\omega}_i^2&\leq
C\int_{M_i}\left( \sum_{j=1}^{N_m} |\lambda^i_j|^2 |\ti{S}^i_j|^2_{\ti{h}_i^m}\right)^{-\frac{\alpha}{m}}\ti{\omega}_i^2\\
&\leq C\int_{M_i}|\ti{S}^i_{N_m}|^{-\frac{2\alpha}{m}}_{\ti{h}_i^m}\ti{\omega}_i^2.\end{split}\]
If we now pick $\alpha<\alpha_{m,1}(M_\infty)$, then by definition we have
$$\int_{M_\infty} |\ti{S}^\infty_{N_m}|^{-\frac{2\alpha}{m}}_{\ti{h}_\infty^m}\ti{\omega}_\infty^2\leq C,$$ and so using Theorem \ref{semic} in the form \eqref{semicc} we get
$$\int_{M_i}e^{-\alpha(\vp_i-\sup_{M_i}\vp_i)}\ti{\omega}_i^2\leq C\int_{M_\infty} |\ti{S}^\infty_{N_m}|^{-\frac{2\alpha}{m}}_{\ti{h}_\infty^m}\ti{\omega}_\infty^2\leq C.$$
Using the complex Monge-Amp\`ere equation \eqref{ma} we get
$$\int_{M_i}e^{-\alpha(\vp_i-\sup_{M_i}\vp_i)}e^{\vp_i-f_i}\omega_i^2\leq C.$$
Since the functions $f_i$ are uniformly bounded (they converge smoothly to $f_\infty$) this implies that
$$\int_{M_i}e^{(1-\alpha)\vp_i+\alpha\sup_{M_i}\vp_i}\omega_i^2\leq C,$$
and applying Jensen's inequality we get
$$\alpha \sup_{M_i}\vp_i +\frac{1-\alpha}{V}\int_{M_i}\vp_i \omega_i^2\leq C,$$
and rearranging
\begin{equation}\label{crucial}
\sup_{M_i}\vp_i\leq \frac{1-\alpha}{\alpha V}\int_{M_i}(-\vp_i)\omega_i^2+C.
\end{equation}
We combine this with the Harnack inequality \eqref{harn1} to get
$$\sup_{M_i}\vp_i\leq -\frac{1-\alpha}{\alpha}\inf_{M_i}\vp_i+C
\leq \frac{2(1-\alpha)}{\alpha}\sup_{M_i}\vp_i+C.$$
All this works as long as $\alpha<\alpha_{m,1}(M_\infty)$. But by assumption this is strictly larger than $2/3$, and so we can choose $\alpha>2/3$ as well. In this case we have that
$$\frac{2(1-\alpha)}{\alpha}<1,$$
and so we immediately get the estimate \eqref{supest}.
\end{proof}
For later use, we collect here what we just proved in \eqref{crucial}
\begin{lemma}\label{good1}
If a partial $C^0$ estimate holds, then for any $0<\alpha<\alpha_{m,1}(M_\infty)$ there is a constant $C>0$ so that for all $i$ we have
\begin{equation}\label{good2}
\sup_{M_i}\vp_i\leq \frac{1-\alpha}{\alpha V}\int_{M_i}(-\vp_i)\omega_i^2+C.
\end{equation}
\end{lemma}
\setcounter{equation}{0}
\section{Step 3 - Improved Harnack Inequality}
In this section we will prove the third and last step in the proof of the main theorem, namely that if a partial $C^0$ estimate holds and if for any $M$ in $\mathfrak{M}_k$ we have $\alpha_{m_0,1}(M)=2/3$ and $\alpha_{m_0,2}(M)>2/3$, then the estimate \eqref{supest} holds. This will complete the proof of the main theorem \ref{main}. Again, we will write $m=m_0$.

The main ingredient is the following improved Harnack inequality:
\begin{proposition}\label{improve}
If a partial $C^0$ estimate holds and if for any $M$ in $\mathfrak{M}_k$ we have $\alpha_{m,2}(M)>2/3$, then there exist $\ve, C>0$ so that for all $i$ we have
\begin{equation}\label{improved}
\frac{1}{V}\int_{M_i}(-\vp_i)\omega_i^2\leq (2-\ve)\sup_{M_i}\vp_i+C.
\end{equation}
\end{proposition}
As an aside, we remark that this indeed improves on the Harnack inequality \eqref{harn1} since we can use the Green formula for the K\"ahler-Einstein metrics $\omega_i$ (as in section 2) to get
\begin{equation}\label{har}
-\inf_{M_i}\vp_i\leq \frac{1}{V}\int_{M_i}(-\vp_i)\omega_i^2+C,
\end{equation}
and so we get 
$$-\inf_{M_i}\vp_i\leq(2-\ve)\sup_{M_i}\vp_i+C,$$
which improves \eqref{harn1}. However, we will only make use of the weaker estimate \eqref{improved}.

If we assume proposition \eqref{improve}, we can complete the proof of Step 3 as follows. If we go back to \eqref{crucial}, or Lemma \ref{good1}, we see that it holds for any 
$0<\alpha<\alpha_{m,1}(M_\infty)=2/3$. Combining \eqref{crucial} with \eqref{improved} we get
$$\sup_{M_i}\vp_i\leq \frac{(1-\alpha)(2-\ve)}{\alpha}\sup_{M_i}\vp_i+C,$$
and if we choose $\alpha$ so that
$$\frac{2-\ve}{3-\ve}<\alpha<\frac{2}{3},$$
we see that
$$\frac{(1-\alpha)(2-\ve)}{\alpha}<1,$$
which immediately implies the estimate \eqref{supest}.

To prove the main theorem it only remains to prove proposition \ref{improve}.
First, we note that the Harnack inequality \eqref{harn1} actually follows from 
\eqref{har} together with the following estimate proved by Tian \cite{Ti2} (again using the fact that there are no nonzero holomorphic vector fields)
\begin{equation}\label{harn2}
\frac{1}{V}\int_{M_i}(-\vp_i)\omega_i^2\leq 2\sup_{M_i}\vp_i-\frac{1}{V}\int_{M_i}\mn\de\vp_i\wedge\db\vp_i\wedge\ti{\omega}_i,
\end{equation}
which improves on \eqref{improved1} since the last term is equal to minus the integral of $|\de\vp_i|^2_{\ti{\omega}_i}$, and so it is nonpositive. Since we will need \eqref{harn2}, let us give an idea of how it is proved. First, using the fact that $M$ has no holomorphic vector fields, one can follow Bando-Mabuchi \cite{BM} and solve Aubin's continuity method backwards, that is for any $0\leq t\leq 1$ one can solve
\begin{equation}\label{aubin}
(\ti{\omega}_i+\mn\de\db\vp_i(t))^2=e^{f_i-t\vp_i(t)}\ti{\omega}_i^2,
\end{equation}
where $\vp_i(t)$ are K\"ahler potentials for $\ti{\omega}_i$, and $\vp_i(1)=\vp_i$. Since from now on all the computations are formal, we will drop the indices $i$ and just call $\ti{\omega}_i=\omega$ and $\ti{\omega}_i+\mn\de\db\vp_i=\omega_\vp$. We also recall the definition of two well-known functionals in dimension $2$
$$I_\omega(\vp)=\frac{1}{V}\int_M \vp(\omega^2-\omega_\vp^2)=\frac{1}{V}\int_M\mn\de\vp\wedge\db\vp\wedge\omega_\vp+\frac{1}{V}\int_M \mn\de\vp\wedge\db\vp\wedge\omega,$$
$$J_\omega(\vp)=\frac{1}{3V}\int_M \mn\de\vp\wedge\db\vp\wedge\omega_\vp+\frac{2}{3V}\int_M \mn\de\vp\wedge\db\vp\wedge\omega,$$
where here $\vp$ is any K\"ahler potential. It is easy to check that for any K\"ahler potential $\vp$ one has
$$I_\omega(\vp)-J_\omega(\vp)\geq 0.$$
If now $\vp$ is the potential $\vp_i$, then using \eqref{aubin} Tian proves (proposition 2.3 in \cite{Ti1}) that
$$\frac{1}{V}\int_M (-\vp)\omega_\vp^2=I_\omega(\vp)-J_\omega(\vp)-\int_0^1 
(I_\omega(\vp(t))-J_\omega(\vp(t)))dt\leq I_\omega(\vp)-J_\omega(\vp).$$
But using the definitions of $I_\omega$ and $J_\omega$ and then integrating by parts, we can write 
$I_\omega(\vp)-J_\omega(\vp)$ as
\[\begin{split} \frac{2}{3V}\int_M\mn\de\vp\wedge\db\vp\wedge(\omega+\omega_\vp)-
\frac{1}{3V}\int_M \mn\de\vp\wedge\db\vp\wedge\omega\\
=\frac{2}{3V}\int_M \vp\omega^2-\frac{2}{3V}\int_M\vp\omega_\vp^2-\frac{1}{3V}\int_M \mn\de\vp\wedge\db\vp\wedge\omega.
\end{split}\]
Putting together the last two equations one gets
\[\begin{split}\frac{1}{3V}\int_M (-\vp)\omega_\vp^2&=\frac{1}{V}\int_M (-\vp)\omega_\vp^2-
\frac{2}{3V}\int_M (-\vp)\omega_\vp^2\\
&\leq \frac{2}{3V}\int_M \vp\omega^2-\frac{1}{3V}\int_M \mn\de\vp\wedge\db\vp\wedge\omega,
\end{split}\]
which multiplied by $3$ gives
\[\begin{split}\frac{1}{V}\int_M (-\vp)\omega_\vp^2&\leq \frac{2}{V}\int_M \vp\omega^2-\frac{1}{V}\int_M \mn\de\vp\wedge\db\vp\wedge\omega\\
&\leq 2\sup_M\vp-\frac{1}{V}\int_M \mn\de\vp\wedge\db\vp\wedge\omega,
\end{split}\]
which is exactly \eqref{harn2} (after reinstating the previous notations).

Before we can prove proposition \ref{improve} we need three more lemmas. For convenience, we now temporarily set 
$$\psi_i=\frac{1}{m}\log\sum_{j=1}^{N_m}|\lambda^i_j|^2 
|\ti{S}^i_j|^2_{\ti{h}_i^m},$$
the algebraic K\"ahler potential, which by the partial $C^0$ estimate \eqref{partiall} satisfies
\begin{equation}\label{harn4}
\sup_i|\vp_i-\sup_{M_i}\vp_i-\psi_i|\leq C_0.
\end{equation}
\begin{lemma} If a partial $C^0$ estimate holds, then there is a constant $C>0$ so that for all $i$ we have
\begin{equation}\label{harn3}
\frac{1}{V}\int_{M_i}(-\vp_i)\omega_i^2\leq 2\sup_{M_i}\vp_i-\frac{1}{V}\int_{M_i}\mn\de\psi_i\wedge\db\psi_i\wedge\ti{\omega}_i+C.
\end{equation}
\end{lemma}
\begin{proof}
Integrating by parts a few times we see that
\[\begin{split}
2\int_{M_i}(\mn\de\vp_i\wedge\db\vp_i\wedge\ti{\omega}_i&-\mn\de\psi_i\wedge\db\psi_i\wedge\ti{\omega}_i)\\
&=\int_{M_i}(\psi_i\Delta_{\ti{\omega}_i}\psi_i-\vp_i\Delta_{\ti{\omega}_i}\vp_i)\ti{\omega}_i^2\\
&=\int_{M_i}(\psi_i-\vp_i)(\Delta_{\ti{\omega}_i}\psi_i+\Delta_{\ti{\omega}_i}\vp_i)\ti{\omega}_i^2\\
&=\int_{M_i}(\psi_i-\vp_i+\sup_{M_i}\vp_i+C_0)(\Delta_{\ti{\omega}_i}\psi_i+\Delta_{\ti{\omega}_i}\vp_i)\ti{\omega}_i^2,
\end{split}\]
where the constant $C_0$ is as in \eqref{harn4}. Since $\psi_i$ and $\vp_i$ are both K\"ahler potentials for $\ti{\omega}_i$ (see \eqref{algpot}), we see that
$$\Delta_{\ti{\omega}_i}\psi_i>-2,\quad \Delta_{\ti{\omega}_i}\vp_i>-2,$$
and since we also have from \eqref{harn4} that
$0\leq \psi_i-\vp_i+\sup_{M_i}\vp_i+C_0\leq 2C_0$, we immediately get
\begin{equation}\label{harn5}
2\int_{M_i}(\mn\de\vp_i\wedge\db\vp_i\wedge\ti{\omega}_i-\mn\de\psi_i\wedge\db\psi_i\wedge\ti{\omega}_i)
\geq -8C_0V.
\end{equation}
For the reverse inequality we compute
\[\begin{split}
2\int_{M_i}(\mn\de\vp_i\wedge\db\vp_i\wedge&\ti{\omega}_i-\mn\de\psi_i\wedge\db\psi_i\wedge\ti{\omega}_i)\\
&=\int_{M_i}(\psi_i\Delta_{\ti{\omega}_i}\psi_i-\vp_i\Delta_{\ti{\omega}_i}\vp_i)\ti{\omega}_i^2\\
&=\int_{M_i}(\vp_i-\psi_i)(-\Delta_{\ti{\omega}_i}\psi_i-\Delta_{\ti{\omega}_i}\vp_i)\ti{\omega}_i^2\\
&=\int_{M_i}(\vp_i-\psi_i-\sup_{M_i}\vp_i+C_0)(-\Delta_{\ti{\omega}_i}\psi_i-\Delta_{\ti{\omega}_i}\vp_i)\ti{\omega}_i^2\\
&\leq 8C_0V,
\end{split}\]
which combined with \eqref{harn5} and \eqref{harn2} gives \eqref{harn3}.
\end{proof}

Now we imitate the proof of Lemma \ref{good1} in Step 2 to get a slightly weaker result in the following way.
\begin{lemma}
If a partial $C^0$ estimate holds, then for any $0<\alpha<\alpha_{m,2}(M_\infty)$ there is a constant $C>0$ so that for all $i$ we have
\begin{equation}\label{good}
\sup_{M_i}\vp_i\leq \frac{1-\alpha}{\alpha V}\int_{M_i}(-\vp_i)\omega_i^2-\frac{2}{m}\log \lambda^i_{N_m-1}+C.
\end{equation}
\end{lemma}
\begin{proof}
For any $\alpha>0$ we compute, using the partial $C^0$ estimate in the form \eqref{partiall}
and the fact that $\lambda^i_{N_m}=1,$
\[\begin{split}\int_{M_i}& e^{-\alpha(\vp_i-\sup_{M_i}\vp_i)}\ti{\omega}_i^2\leq
C\int_{M_i}\left( \sum_{j=1}^{N_m} |\lambda^i_j|^2 |\ti{S}^i_j|^2_{\ti{h}_i^m}\right)^{-\frac{\alpha}{m}}\ti{\omega}_i^2\\
&\leq C|\lambda^i_{N_m-1}|^{-\frac{2\alpha}{m}}\int_{M_i}\left(|\ti{S}^i_{N_m-1}|^2_{\ti{h}_i^m}+
|\ti{S}^i_{N_m}|^2_{\ti{h}_i^m}\right)^{-\frac{\alpha}{m}}\ti{\omega}_i^2.\end{split}\]
Recall that the sections $\ti{S}^i_j$ converge smoothly to holomorphic sections $\ti{S}^\infty_j$ on $M_\infty$, which form an orthonormal basis for the $L^2$ inner product defined using $\ti{h}_\infty^m$ and $\ti{\omega}_\infty^2$.
If we now pick $\alpha<\alpha_{m,2}(M_\infty)$, then by definition we have
$$\int_{M_\infty} \left(|\ti{S}^\infty_{N_m-1}|^2_{\ti{h}_\infty^m}+
|\ti{S}^\infty_{N_m}|^2_{\ti{h}_\infty^m}\right)^{-\frac{\alpha}{m}}\ti{\omega}_\infty^2\leq C,$$ 
and using Demailly-Koll\'ar's generalization of theorem \ref{semic} \cite{DK} we have
\[\begin{split}\lim_{i\to\infty} \int_{M_i}&\left(|\ti{S}^i_{N_m-1}|^2_{\ti{h}_i^m}+
|\ti{S}^i_{N_m}|^2_{\ti{h}_i^m}\right)^{-\frac{\alpha}{m}}\ti{\omega}_i^2\\
&=\int_{M_\infty} \left(|\ti{S}^\infty_{N_m-1}|^2_{\ti{h}_\infty^m}+
|\ti{S}^\infty_{N_m}|^2_{\ti{h}_\infty^m}\right)^{-\frac{\alpha}{m}}\ti{\omega}_\infty^2\leq C,\end{split}\]
and so we get
$$\int_{M_i}e^{-\alpha(\vp_i-\sup_{M_i}\vp_i)}\ti{\omega}_i^2\leq C|\lambda^i_{N_m-1}|^{-\frac{2\alpha}{m}}.$$
Using the complex Monge-Amp\`ere equation \eqref{ma} we get
$$\int_{M_i}e^{-\alpha(\vp_i-\sup_{M_i}\vp_i)}e^{\vp_i-f_i}\omega_i^2\leq C|\lambda^i_{N_m-1}|^{-\frac{2\alpha}{m}}.$$
The functions $f_i$ are uniformly bounded and we can apply Jensen's inequality to get
$$\alpha \sup_{M_i}\vp_i +\frac{1-\alpha}{V}\int_{M_i}\vp_i \omega_i^2\leq -\frac{2\alpha}{m}\log \lambda^i_{N_m-1}+C,$$
which is exactly \eqref{good}.
\end{proof}
Notice now that if $$\sup_i (-\log\lambda^i_{N_m-1})\leq C,$$
then from \eqref{good} we get exactly the same estimate as in \eqref{crucial}, and using the fact that 
we can choose $2/3<\alpha<\alpha_{m,2}(M_\infty)$ and immediately conclude that \eqref{supest} holds, as in Step 2. So, up to subsequence, we are free to assume that
\begin{equation}\label{limit}
\lim_{i\to \infty}\lambda^i_{N_m-1}=0.
\end{equation}

The next step is the following
\begin{lemma}
If a partial $C^0$ estimate holds, then there are constants $C>0$ and $0<\delta<1$ so that for all $i$ we have
\begin{equation}\label{annoy}\frac{1}{V}\int_{M_i}\mn\de\psi_i\wedge\db\psi_i\wedge\ti{\omega}_i\geq-\frac{2\delta}{m}\log\lambda^i_{N_m-1}-C.
\end{equation}
\end{lemma}
This bound is very crude; with more care it is possible to see that one can choose $\delta$ as close to $1$ as one wants, at the expense of enlarging $C$. However, any $\delta>0$ will be enough for us.
\begin{proof}
First of all, from the definition of $\psi_i=\frac{1}{m}\log\sum_{j=1}^{N_m}|\lambda^i_j|^2 
|\ti{S}^i_j|^2_{\ti{h}_i^m},$ we have
\begin{equation}\label{long}
\begin{split}\int_{M_i}\mn\de\psi_i\wedge\db\psi_i\wedge\ti{\omega}_i&=\frac{1}{m^2}\int_{M_i}
\left|\de \log\sum_{j=1}^{N_m}|\lambda^i_j|^2 
|\ti{S}^i_j|^2_{\ti{h}_i^m}\right|^2_{\ti{\omega}_i}\ti{\omega}_i^2\\
&=\frac{1}{m^2}\int_{M_i}
\frac{\left|\de \sum_{j=1}^{N_m}|\lambda^i_j|^2 
|\ti{S}^i_j|^2_{\ti{h}_i^m}\right|^2_{\ti{\omega}_i}}{\left|\sum_{j=1}^{N_m}|\lambda^i_j|^2 
|\ti{S}^i_j|^2_{\ti{h}_i^m}\right|^2}\ti{\omega}_i^2.\end{split}
\end{equation}
We now pick a point $x_i\in M_i$ where the section $\ti{S}^i_{N_m}$ vanishes, but where none of the other sections 
$\ti{S}^i_{j}, j<N_m$ vanishes. Such a point exists because the sections are linearly independent. We can moreover assume that the zero locus of $\ti{S}^i_{N_m}$, which will be denoted by $\Sigma_i$, is smooth near $x_i$ (disregarding multiplicities), and near $x_i$ there is a chart with holomorphic coordinates $(z,w)$ (which depend on $i$, but converge to holomorphic coordinates on $M_\infty$) centered at $x_i$ so that locally 
\begin{equation}\label{brut}
|\ti{S}^i_{N_m}|^2_{\ti{h}_i^m}=|z|^{2\ell} F,
\end{equation}
where $F$ is a smooth function that depends on $i$ (but is bounded in $C^\infty$ uniformly in $i$) and $F$ does not vanish on $\Sigma_i$. Here $\ell$ is a positive integer (which depends on $i$ but is bounded), which is the order of vanishing of $\ti{S}^i_{N_m}$ along $\Sigma_i$ near $x_i$.

We fix a small radius $r$ so that if $B_i=\{|z|\leq r,\ |w|\leq r\}$, then on $B_i$ we have
$$0<\frac{1}{C}\leq |\ti{S}^i_j|^2_{\ti{h}_i^m}\leq C, \quad \left|\de  |\ti{S}^i_j|^2_{\ti{h}_i^m}\right|^2_{\ti{\omega}_i}\leq C,$$
for all $1\leq j<N_m$ (the fact that we can do this with $C$ independent of $i$ follows from the fact that the metrics $\ti{\omega}_i$ and the sections $\ti{S}^i_j$ converge smoothly). Notice that 
$\Sigma_i\cap B_i=\{z=0\}\cap B_i$. 
We then have
$$\left|\sum_{j=1}^{N_m}|\lambda^i_j|^2 
|\ti{S}^i_j|^2_{\ti{h}_i^m}\right|^2\leq \left(C|\lambda^i_{N_m-1}|^2+|\ti{S}^i_{N_m}|^2_{\ti{h}_i^m}\right)^2,$$
$$\left|\de \sum_{j=1}^{N_m}|\lambda^i_j|^2 
|\ti{S}^i_j|^2_{\ti{h}_i^m}\right|^2_{\ti{\omega}_i}\geq \frac{1}{C}\left|\de  |\ti{S}^i_{N_m}|^2_{\ti{h}_i^m}\right|^2_{\ti{\omega}_i}-C|\lambda^i_{N_m-1}|^4.$$
Recall that from \eqref{limit} we can assume that $\lambda^i_{N_m-1}$ goes to zero. 
We then let 
$$B'_i=\{|z|\leq (\lambda^i_{N_m-1})^{1/\ell},\ |w|\leq r\},$$ which is a small neighborhood of $\Sigma_i$ contained inside $B_i$ (for $i$ large).
We can then bound the last integral in \eqref{long} by
\begin{equation}\label{pr1}
\begin{split}\frac{1}{m^2}\int_{B_i\backslash B'_i}&\frac{\left|\de \sum_{j=1}^{N_m}|\lambda^i_j|^2 
|\ti{S}^i_j|^2_{\ti{h}_i^m}\right|^2_{\ti{\omega}_i}}{\left|\sum_{j=1}^{N_m}|\lambda^i_j|^2 
|\ti{S}^i_j|^2_{\ti{h}_i^m}\right|^2}\ti{\omega}_i^2\\
&\geq
\frac{1}{C}\int_{B_i\backslash B'_i}\frac{\left|\de  |\ti{S}^i_{N_m}|^2_{\ti{h}_i^m}\right|^2_{\ti{\omega}_i}-C|\lambda^i_{N_m-1}|^4}
{\left(C|\lambda^i_{N_m-1}|^2+|\ti{S}^i_{N_m}|^2_{\ti{h}_i^m}\right)^2}\ti{\omega}_i^2\\
&\geq \frac{1}{C}\int_{B_i\backslash B'_i}\frac{\left|\de  |\ti{S}^i_{N_m}|^2_{\ti{h}_i^m}\right|^2_{\ti{\omega}_i}}{\left(C|\lambda^i_{N_m-1}|^2+|\ti{S}^i_{N_m}|^2_{\ti{h}_i^m}\right)^2}\ti{\omega}_i^2-C.
\end{split}\end{equation}
Since on $B_i\backslash B'_i$ we have $|\ti{S}^i_{N_m}|^2_{\ti{h}_i^m}=|z|^{2\ell} F,$ on the same region we have
$$\left|\de  |\ti{S}^i_{N_m}|^2_{\ti{h}_i^m}\right|^2_{\ti{\omega}_i}\ti{\omega}_i^2\geq \frac{1}{C}|z|^{4\ell-2}dV_E,$$
since $|z|\leq r$ is small, where $dV_E=\mn dz\wedge d\ov{z}\wedge\mn dw\wedge d\ov{w}$ is the Euclidean volume form. We can then estimate
\begin{equation}\label{pr2}
\begin{split}
\int_{B_i\backslash B'_i}&\frac{\left|\de  |\ti{S}^i_{N_m}|^2_{\ti{h}_i^m}\right|^2_{\ti{\omega}_i}}
{\left(C|\lambda^i_{N_m-1}|^2+|\ti{S}^i_{N_m}|^2_{\ti{h}_i^m}\right)^2}\ti{\omega}_i^2\\
&\geq
\frac{1}{C}\int_{|w|\leq r}\int_{(\lambda^i_{N_m-1})^{1/\ell}\leq |z|\leq r}
\frac{|z|^{4\ell-2}}{|\lambda^i_{N_m-1}|^4+|z|^{4\ell}}dV_E\\
&\geq \frac{1}{C}\int_{(\lambda^i_{N_m-1})^{1/\ell}}^{r} \frac{\rho^{4\ell-2}}{|\lambda^i_{N_m-1}|^4
+\rho^{4\ell}}\rho d\rho\\
&=\frac{1}{C}\log\left(\frac{|\lambda^i_{N_m-1}|^4+r^{4\ell}}{|\lambda^i_{N_m-1}|^4+|\lambda^i_{N_m-1}|^4}\right)\\
&\geq -\frac{1}{C}\log\lambda^i_{N_m-1}-C,
\end{split}\end{equation}
for a uniform constant $C$. 
Combining the estimates \eqref{pr1} and \eqref{pr2} with \eqref{long}
finally proves \eqref{annoy}, for a suitable uniform $\delta>0$.
\end{proof}

Finally we can prove proposition \ref{improve}, which will finish the proof of the main theorem \ref{main}.
\begin{proof}[Proof of proposition \ref{improve}]
First of all we combine \eqref{harn3} with \eqref{annoy} and get
\begin{equation}\label{end1}\begin{split}\frac{1}{V}\int_{M_i}(-\vp_i)\omega_i^2&\leq 2\sup_{M_i}\vp_i-\frac{1}{V}\int_{M_i}\mn\de\psi_i\wedge\db\psi_i\wedge\ti{\omega}_i+C\\
&\leq 2\sup_{M_i}\vp_i+\frac{2\delta}{m}\log\lambda^i_{N_m-1}+C.\end{split}\end{equation}
On the other hand \eqref{good} gives
\begin{equation}\label{end2}\frac{2\delta}{m}\log \lambda^i_{N_m-1}\leq -\delta\sup_{M_i}\vp_i+\frac{\delta(1-\alpha)}{\alpha V}\int_{M_i}(-\vp_i)\omega_i^2+C,\end{equation}
and combining \eqref{end1} and \eqref{end2} we get
\begin{equation}\label{end3}
\frac{1}{V}\left(1-\frac{\delta(1-\alpha)}{\alpha}\right)\int_{M_i}(-\vp_i)\omega_i^2
\leq (2-\delta)\sup_{M_i}\vp_i+C.
\end{equation}
Since we are assuming that $\alpha_{m,2}(M_\infty)>2/3$, we now choose $\alpha$ so that
$2/3<\alpha<\alpha_{m,2}(M_\infty)$.
Then, since $\delta<1$ and $\alpha>2/3$ we see that the coefficient $(1-\frac{\delta(1-\alpha)}{\alpha})$ is positive (in any case we could have just taken a smaller $\delta$), and so from \eqref{end3} we get
$$\frac{1}{V}\int_{M_i}(-\vp_i)\omega_i^2
\leq \frac{\alpha(2-\delta)}{\alpha(1+\delta)-\delta}\sup_{M_i}\vp_i+C.$$
But since $\alpha>2/3$ one immediately checks that
$$\frac{\alpha(2-\delta)}{\alpha(1+\delta)-\delta}=2-\ve,$$
with 
$$\ve=\frac{3\alpha-2}{\frac{\alpha}{\delta}+\alpha-1}>0,$$
which completes the proof of \eqref{improved}.
\end{proof}
\setcounter{equation}{0}
\section{Appendix - Algebraic and Analytic $\alpha$-invariants}
In this appendix we will give an idea of why it is possible to compute the $\alpha$ invariant of a Fano manifold using algebraic geometry.

In what follows, $M$ will be an $n$-dimensional compact K\"ahler manifold with an ample line bundle $L$. We fix $\omega$ a K\"ahler metric in $c_1(L)$, and define Tian's $\alpha$ invariant
\[\begin{split}\alpha(L)=\sup\bigg\{\alpha>0\ &\bigg|\ \exists C>0 \textrm{ with }\int_M e^{-\alpha(\vp-\sup_M\vp)}\omega^n\leq C,\\ &\textrm{ for all }\vp\in C^\infty(M,\mathbb{R})\textrm{ with }\omega+\mn\de\db\vp>0\bigg\}.\end{split}\]
If $M$ is Fano and $L=K_M^{-1}$ this is exactly our previous definition.

We say that an $L^1$ function $\vp$ is $\omega-PSH$ if it is u.s.c. and it satisfies
$\omega+\mn\de\db\vp\geq 0$ in the sense of distributions. For any such $\vp$ we define its complex singularity exponent as
$$c(\vp)=\sup\bigg\{\alpha>0\ \bigg|\ \int_M e^{-\alpha\vp}\omega^n<\infty\bigg\}.$$
The main result of this section is the following proposition
\begin{proposition}[Demailly]We have that
$$\alpha(L)=\inf\{c(\vp)\ |\ \vp\textrm{ is } \omega-PSH\}.$$
\end{proposition}
\begin{proof}
For convenience let us denote $\ti{\alpha}(L)=\inf\{c(\vp)\ |\ \vp\textrm{ is } \omega-PSH\}.$
We fist show that $\alpha(L)\leq \ti{\alpha}(L)$. If not, we can find a number $\alpha$ with
$\ti{\alpha}(L)<\alpha<\alpha(L)$, so from the definitions on the one hand we have that there is $C>0$ with
\begin{equation}\label{uno}
\int_M e^{-\alpha(\vp-\sup_M\vp)}\omega^n\leq C,
\end{equation}
for all K\"ahler potentials $\vp$, but on the other hand there exists $\vp$ which is only $\omega-PSH$ so that
$$\int_M e^{-\alpha\vp}\omega^n=+\infty.$$
Since $\vp$ is u.s.c., it is bounded above, and this together with the fact that $\vp$ is in $L^1$ imply that $-\infty<\sup_M\vp<\infty$.
We apply a special case of Demailly's regularization theorem \cite{D1} (see also B\l ocki-Ko\l odziej \cite{BK} for a short proof)
and we see that there exist smooth functions $\vp_i$ with $\omega+\mn\de\db\vp_i>0$ (i.e. K\"ahler potentials) that decrease pointwise to $\vp$. In particular $\sup_M\vp_i$ is bounded uniformly for $i$ large, so from \eqref{uno} we see that
\begin{equation}\label{due}
\int_M e^{-\alpha\vp_i}\omega^n\leq C,
\end{equation}
for all $i$ large. Since the functions $e^{-\alpha\vp_i}$ increase to $e^{-\alpha\vp}$, it follows from the Lebesgue monotone convergence theorem that
$$\lim_{i\to\infty}\int_M e^{-\alpha\vp_i}\omega^n=\int_M e^{-\alpha\vp}\omega^n=+\infty,$$
which contradicts \eqref{due}. So $\alpha(L)\leq\ti{\alpha}(L)$.

On the other hand, if $\alpha(L)<\ti{\alpha}(L)$ then we can find $\alpha$ with 
$\alpha(L)<\alpha<\ti{\alpha}(L)$, so from the definitions we have that for any $\omega-PSH$ function the complex singularity exponent satisfies $c(\vp)>\alpha$, and so 
$$\int_M e^{-\alpha\vp}\omega^n<\infty,$$
but on the other hand there exist smooth K\"ahler potentials $\vp_i$ with
\begin{equation}\label{tre}
\int_M e^{-\alpha(\vp_i-\sup_M\vp_i)}\omega^n \geq i.
\end{equation}
By weak compactness of the currents $\omega+\mn\de\db\vp_i$, modulo subsequence we can assume that the functions $\vp_i-\sup_M\vp_i$ converge in $L^1$ to a limit $\psi$ which is $\omega-PSH$. Since we know that $c(\psi)>\alpha$, a theorem of Demailly-Koll\'ar \cite[Theorem 0.2 (2)]{DK} (which generalizes theorem \ref{semic}) implies that the functions $e^{-\alpha(\vp_i-\sup_M\vp_i)}$ converge in $L^1$ to $e^{-\alpha\psi}$. Since $\int_M e^{-\alpha\psi}\omega^n<\infty$, this contradicts \eqref{tre}.
\end{proof}

To relate this to algebraic geometry, let $D$ be any nonzero divisor in the linear series $|mL|$.
Therefore there is a global nonzero holomorphic section $S$ of $L^m$ with zero divisor equal to $D$.
Since $S$ is unique up to scaling, we will rescale it so that
$$\int_M |S|^2_{h^m}\omega^n=1,$$
where $h$ is a metric on $L$ with curvature $\omega$.
We then define the (global) log canonical threshold of the divisor $\frac{1}{m}D$ by
$$\mathrm{lct}\left(\frac{1}{m}D\right)=c\left(\frac{1}{m}\log|S|^2_{h^m}\right),$$
where notice that the function $\frac{1}{m}\log|S|^2_{h^m}$ is indeed $\omega-PSH$ because of the Poincar\'e-Lelong formula: $\omega+\frac{1}{m}\mn\de\db\log|S|^2_{h^m}=\frac{1}{m}[D]$, where $[D]$ is the current of integration along $D$.

The number $\mathrm{lct}\left(\frac{1}{m}D\right)$ can be computed in a purely algebraic way, and it depends on the singularities of $D$ (see \cite{CS}).
In the appendix of \cite{CS} Demailly proved the following result (see also \cite{DK}):
\begin{theorem}[Demailly]\label{dema}
We have
$$\alpha(L)=\inf_{m\geq 1}\inf_{D\in|mL|}\mathrm{lct}\left(\frac{1}{m}D\right).$$
\end{theorem}
We refer to that paper for the proof, that relies crucially on the Ohsawa-Takegoshi extension theorem.
Using this result, the computation of the $\alpha$ invariant is reduced to computing log canonical thresholds of divisors. This is the approach taken by Cheltsov \cite{Ch} to prove Theorem \ref{estima} (see also \cite{Sh, CW}).

Finally, we remark that with similar arguments one can characterize also the invariants $\alpha_{m,1}(L), \alpha_{m,2}(L)$ as follows (see \cite{Sh})
$$\alpha_{m,1}(L)=\inf_{D\in|mL|}\mathrm{lct}\left(\frac{1}{m}D\right),$$
\[\begin{split}\alpha_{m,2}(L)=\inf\bigg\{c\left(\frac{1}{m}\log(|S_1|^2_{h^m}+|S_2|^2_{h^m})\right)\ 
\bigg|&\ S_1,S_2\in H^0(L^{m}),\\
& \int_M \langle S_i,S_j\rangle_{h^m}\omega^n=\delta_{ij}\bigg\},\end{split}\]
where $c\left(\frac{1}{m}\log(|S_1|^2_{h^m}+|S_2|^2_{h^m})\right)$ can also be interpreted algebraically as the log canonical threshold of $\frac{1}{m}\mathscr{I}$, where $\mathscr{I}$ is the ideal sheaf generated by $S_1, S_2$.
Notice that from theorem \ref{dema} it follows that
$$\alpha(L)=\inf_{m\geq 1}\alpha_{m,1}(L),$$
and in fact Tian conjectured \cite[Question 1]{Ti4} that when $m$ is large the numbers $\alpha_{m,1}(L)$ stabilize to $\alpha(M)$. 
On the other hand one also has that 
$$\alpha(L)=\inf_{m\geq 1}\alpha_{m,2}(L),$$ 
but it is known that these to not stabilize to $\alpha(M)$ in general, if $M$ is allowed to have rational double points singularities \cite[Remark 1.7]{kos}.

\end{document}